\documentclass[a4paper,12pt]{article}
\usepackage{amsmath}
\usepackage{amssymb}
\usepackage{setspace}
\usepackage{fullpage}
\usepackage{pgf,tikz}
\usetikzlibrary{arrows}
\usepackage{bbm}
\usepackage{amsthm}
\newtheorem{theorem}{Theorem}[section]
\newtheorem{lemma}[theorem]{Lemma}

\newtheorem{corollary}[theorem]{Corollary}

\newtheorem{remarkN}{Remark}
\theoremstyle{definition}
\newtheorem{construction}{Construction}
\newtheorem{proposition}{Proposition}

\usepackage{pdfsync}
\def\PG{\mathrm{PG}}

 \def\B{\mathcal{B}} 
\def\D{\mathcal{D}}  \def\G{\mathcal{G}}

 \def\Q{\mathcal{Q}}
 \def\S{\mathcal{S}}

\def\F{\mathbb{F}}

\title{Constructing minimal blocking sets using field reduction}
\author{ Geertrui Van de Voorde\thanks{This author is supported by the Fund for Scientific Research Flanders
(FWO -- Vlaanderen).}}
\begin{document}
\maketitle
\begin{abstract} We present a construction for minimal blocking sets with respect to $(k-1)$-spaces in $\PG(n-1,q^t)$, the $(n-1)$-dimensional projective space over the finite field $\F_{q^t}$ of order $q^t$. The construction relies on the use of blocking cones in the {\em field reduced} representation of $\PG(n-1,q^t)$, extending the well-known construction of linear blocking sets. This construction is inspired by the construction for minimal blocking sets with respect to the hyperplanes by Mazzocca, Polverino and Storme ({\em the MPS-construction}); we show that for a suitable choice of the blocking cone over a planar blocking set, we obtain larger blocking sets than the ones obtained from planar blocking sets in \cite{pol}.

Furthermore we show that  every minimal blocking set with respect to the hyperplanes in $\PG(n-1,q^t)$ can be obtained by applying field reduction to a minimal blocking set with respect to $(nt-t-1)$-spaces in $\PG(nt-1,q)$. We end by relating these constructions to the linearity conjecture for small minimal blocking sets. We show that if a small minimal blocking set is constructed from the MPS-construction, it is of R\'edei-type whereas a small minimal blocking set arises from our cone construction if and only if it is linear.
\end{abstract}

\section{Introduction}

This paper is inspired by the paper \cite{Leo}, where Mazzocca, Polverino and Storme construct minimal blocking sets with respect to the hyperplanes in $\PG(n,q^t)$ by using certain cones in the Barlotti-Cofman representation of $\PG(n,q^t)$, extending the results of \cite{pol} to general dimension. Our paper is organised as follows. In Section \ref{barlotti}, we give the necessary background on the Barlotti-Cofman and field reduced representation of $\PG(n,q)$ and recall the correspondence between these representations. In Section \ref{S3}, we explain how the construction of Mazzocca, Polverino and Storme (also called the {\em MPS-construction}) can be presented in an easier way by making use of field reduction: the obtained blocking set corresponds to the points of a minimal blocking set (with respect to subspaces of a particular dimension) in a projective space over $\F_q$, considered over $\F_{q^t}$. We also show that the construction of linear blocking sets and the recent construction by Costa \cite{costa} fit in this framework. By using cones in the field reduced representation of $\PG(n,q^t)$ we generalise the MPS-construction in Section \ref{S4}. Starting from a planar blocking set, we construct non-planar blocking sets with respect to $(k-1)$-spaces in $\PG(n-1,q^t)$. In Corollary \ref{crr}, we show that if we choose the defining blocking cone carefully, we construct blocking sets whose size exceeds the ones obtained from the MPS-construction using planar blocking sets.

Finally, in Section \ref{hyper}, we show that {\em every} minimal blocking set with respect to the hyperplanes in $\PG(n-1,q^t)$ can be obtained by applying field reduction to a minimal blocking set with respect to $(nt-t-1)$-spaces in $\PG(nt-1,q)$. This also provides us with a different view on the linearity conjecture for small minimal blocking sets. Finally, we show that if a small minimal blocking set is obtained by the MPS-construction then it is  of R\'edei-type, whereas a small minimal blocking set arises from our Construction \ref{con1} if and only if it is a linear blocking set. 

\section{The Barlotti-Cofman representation and field reduction}\label{barlotti}
\subsection{Desarguesian spreads and field reduction}
Throughout this paper, we let $\PG(m-1,q)$ denote the $(m-1)$-dimensional projective space over the finite field $\F_q$ of order $q$. A {\em $(t-1)$-spread} of a projective space $\PG(m-1,q)$ is a family of mutually disjoint subspaces of dimension $(t-1)$ partitioning the space $\PG(m-1,q)$. It is not hard to show that if a $(t-1)$-spread of $\PG(m-1,q)$ exists, then $t$ divides $m$. On the other hand, if $t$ divides $m$, there exists a $(t-1)$-spread of $\PG(m-1,q)$. This was already shown by Segre \cite{segre}, and can also be seen as follows.

By {\em field reduction} every point of $\PG(n-1,q^t)$ corresponds to a $1$-dimensional vector space over $\F_{q^t}$, which is a $t$-dimensional vector space over $\F_q$, and hence, also corresponds to a projective $(t-1)$-dimensional space over $\F_q$. The set of all $(t-1)$-spaces obtained in this way forms a spread of $\PG(nt-1,q)$, which is called a {\em Desarguesian} $(t - 1)$-spread. Throughout this paper, this $(t-1)$-spread in $\PG(nt-1,q)$ is fixed and is denoted by $\D$. A {\em $\D$-subspace} of $\PG(nt-1,q)$ is a space spanned by elements of $\D$. It follows from the construction that a $\D$-subspace is partitioned by elements of $\D$ and corresponds to a field reduced subspace $\pi$ of $\PG(n-1,q^t)$. If the subspace $\pi$ has dimension $r-1$, then we say that the $\D$-subspace of dimension $rt-1$ corresponding to $\pi$ is a $\D_{r-1}$-subspace.

The following statement is well-known and can be proven by a straightforward counting argument. It will be of use later in this paper.
\begin{lemma} \label{hyperspan} Every hyperplane of $\PG(nt-1,q)$ contains exactly one $\D_{n-2}$-subspace, i.e. an $(nt-t-1)$-space spanned by elements of $\D$. 
\end{lemma}
If $U$ is a subset of $\PG(nt-1,q)$, then we define $\B(U) := \{R \in \D \mid U \cap R \neq \emptyset\}$. In this paper, we identify the elements of $\B(U)$ with their corresponding points of $\PG(n-1, q^t)$. Linear sets can be defined in several equivalent ways, but using the terminology of this paper, an {\em $\F_q$-linear set $S$} in $\PG(n-1,q^{t})$ is a set of points such that $S=\B(\mu)$, where $\mu$ is a subspace of $\PG(nt-1,q)$. For more information on field reduction and linear sets, we refer to \cite{LaVa13}.

\subsection{The Barlotti-Cofman representation}

Let $H$ be a hyperplane of $\PG(n-1,q^t)$; by field reduction $H$ corresponds to a $\D_{n-2}$-subspace $\Sigma$ of $\PG(nt-1,q)$. Note that $\Sigma$ has dimension $nt-t-1$. Let $\Sigma'$ be an $(nt-t)$-space through $\Sigma$ in $\PG(nt-1,q)$.

Consider the following geometry $\Pi_{n-1}=\Pi_{n-1}(\Sigma',\Sigma,\S)$, where $\S$ is the set of elements of $\D$ contained in $\Sigma$:

\begin{itemize}
\item Points: the points of $\Sigma'\setminus \Sigma$ and the elements of $\S$.
\item Lines: the $t$-subspaces of $\Sigma'$ meeting $\Sigma$ exactly in an element of $\S$, together with the $\D_1$-subspaces contained in $\Sigma$.
\item Incidence: natural.
\end{itemize}

The incidence structure $\Pi_{n-1}$ is isomorphic to the design obtained by taking points and lines of $\PG(n-1,q^t)$ and we say that $\Pi_{n-1}$ is the {\em Barlotti-Cofman representation} of $\PG(n-1,q^t)$ \cite{Barlotti}. 

\subsection{The correspondence between the Barlotti-Cofman representation and the representation using field reduction}

From the definitions, we get that the geometry $\mathcal{G}$ with as points the spread elements of $\D$ and as lines the $\D_1$-spaces of $\PG(nt-1,q)$ is isomorphic to the design of points and lines of $\PG(n-1,q^t)$. Let, as in the previous section, $\Sigma'$ be an $(nt-t)$-space through the $\D_{n-2}$-space $\Sigma$ in $\PG(nt-1,q)$. 

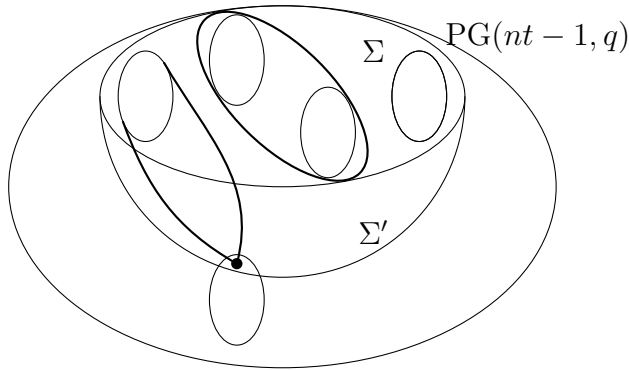
\begin{figure}[h]
\begin{center}
\begin{tikzpicture}[scale=0.6]
\draw (0,0) ellipse (6cm and 4cm);
\draw (0,2) ellipse (4cm and 2cm);
\draw(-4,2)  arc  (180:360:4cm and 4cm);
\draw (-3,2) ellipse (0.6cm and 1cm);
\draw (-1,2.8) ellipse (0.6cm and 1cm);
\draw (1,1.2) ellipse (0.6cm and 1cm);
\draw (3,2) ellipse (0.6cm and 1cm);
\draw (-1,-2.5) ellipse (0.6cm and 1cm);
\draw (3,2) ellipse (0.6cm and 1cm);
\draw[thick,rotate around={135:(0,2)}]  (0,2) ellipse (2.4cm and 1.1cm);

\draw[thick] (-1,-1.7) to [out=75, in=300] (-2.6,2.75);
\draw[thick] (-1,-1.7) to [out=150,in=290] (-3.5,1.45);

\node at (5.6,3.3) { $\mathrm{PG}(nt-1,q)$};
\node at (2,3) {  $\Sigma$};
\node at (2,-1) {  $\Sigma'$};
\node[fill=black,circle,inner sep=1.5pt] at (-1,-1.7) {};
\end{tikzpicture}

\caption{The Barlotti-Cofman representation inside the field reduction representation}
\end{center}

\end{figure}

It is clear that the Barlotti-Cofman representation $\Pi_{n-1}$ and $\mathcal{G}$ are isomorphic: consider the following mapping $\phi$

\begin{eqnarray*}
\phi: \G &\rightarrow& \Pi_{n-1}\\
R\not\subset \Sigma &\mapsto& R\cap \Sigma' \mbox{ for }R\in \D\\
R\subset \Sigma &\mapsto& R \mbox{ for }R\in \D\\
L \not\subset \Sigma &\mapsto& L\cap \Sigma',\mbox{ for } L \mbox{ a $\D_1$-space}\\
L \subset \Sigma &\mapsto& L, \mbox{ for } L \mbox{ a $\D_1$-space}.
\end{eqnarray*}

It is easy to see that $\phi$ defines an isomorphism between $\mathcal{G}$ and $\Pi_{n-1}$. This isomorphism will enable us to describe the MPS-construction in an easier way.

\section{Blocking sets}\label{S3}
A {\em blocking set} $B$ in $\PG(n,q)$ {\em with respect to $k$-spaces} is a set of points such that every $k$-dimensional space in $\PG(n,q)$ (or {\em $k$-space}) contains at least one point of $B$. We also say that the set $B$ {\em blocks} all $k$-spaces. If we are considering blocking sets with respect to the hyperplanes, we simply say that $B$ is a blocking set. A {\em minimal} blocking set $B$ (w.r.t. $k$-spaces) is a blocking set such that no proper subset of $B$ is a blocking set (w.r.t. $k$-spaces). An {\em essential point} of a blocking set with respect to $k$-spaces $B$ is a point lying on a tangent $k$-space to $B$ and we see that $B$ is minimal if and only if every point of $B$ is essential. A blocking set w.r.t. $k$-spaces in $\PG(n,q)$ is called {\em trivial} if it contains an $(n-k)$-space. 
A {\em small} blocking set in $\PG(n,q)$ with respect to $k$-spaces is a blocking set of size smaller than $3(q^{n-k}+1)/2$. A blocking set $B$ (w.r.t. hyperplanes) in $\PG(n,q)$ is of {\em R\'edei-type} if there exists a hyperplane meeting $B$ in $\vert B\vert-q$ points.

Most constructions for blocking sets concern the planar case or the case of blocking sets with respect to the hyperplanes in $\PG(n,q)$. For blocking sets with respect to $k$-spaces, $k\neq n-1$, there are results of Beutelspacher \cite{beutelspacher}  and Heim \cite{heim} classifying the smallest non-trivial blocking sets as cones with base a blocking set in a plane. Other results characterise the smallest blocking sets that span a space of a certain fixed dimension \cite{bokler} or aim at classifying small minimal blocking sets as {\em linear sets} (see Section \ref{hyper}).
\subsection{The cone construction for blocking sets}
We recall that the {\em cone} $K$ with {\em vertex} $\Omega$, where $\Omega$ is a subspace of $\PG(n,q)$ and {\em base} $\bar{B}$, contained in a subspace $\Gamma$, skew from $\Omega$, is the set $\cup_{\bar{P}\in \bar{B}}\langle \bar{P},\Omega\rangle.$

The following lemma is well-known, but since we did not find an exact reference, we give a proof for completeness.

\begin{lemma}\label{kegel} Let $\Omega$ be an $s$-dimensional subspace of $\PG(n,q)$, let $\Gamma$ be an $(n-s-1)$-space disjoint from $\Omega$. The set $\bar{B}$ is a minimal blocking set with respect to $k$-spaces of the space $\Gamma=\PG(n-s-1,q)$ ($k<n-s-1$), if and only if the cone $K$ with vertex $\Omega$ and base $\bar{B}$ is a minimal blocking set with respect to $k$-spaces of $\PG(n,q)$. 
\end{lemma}
\begin{proof} First assume that $\bar{B}$ is a minimal blocking set with respect to $k$-spaces of $\Gamma$. Let $\mu$ be a $k$-space of $\PG(n,q)$. If $\mu$ meets $\Omega$, then $K$ blocks $\Omega$, so assume that $\mu$ is skew from $\Omega$. The projection of $\mu$ from $\Omega$ onto $\Gamma$ (i.e. $\langle \Omega,\mu\rangle \cap \Gamma$) is a $k$-space $\mu'$, which contains at least one point $P$ of $\bar{B}$. This implies that $\mu$ meets $\langle \Omega, P\rangle$ in at least one point, hence, $\mu$ contains a point of $K$. This shows that $K$ is a blocking set. We will now show that $K$ is minimal. Since $\bar{B}$ is a minimal blocking set with respect to $k$-spaces in $\Gamma$, every point $Q$ of $\bar{B}$ lies on a tangent $k$-space $T_Q$. Let $\nu$ be a $(k-1)$-space in $T_Q$, not through $Q$, then for every point $R$ of $\Omega$, $\langle R,\nu\rangle$ is a tangent $k$-space to $K$ through the point $R$, so every point of $\Omega$ is essential. Now let $S$ be a point of $K$, not in $\Omega$. The projection of $S$ onto $\Gamma$ is a point $S'$ of $\bar{B}$, lying on a tangent $k$-space $T_{S'}$. The space $\langle \Omega, T_{S'}\rangle$ is $(k+s+1)$-dimensional, hence, we can take a $k$-dimensional subspace of $\langle \Omega, T_{S'}\rangle$, meeting $\langle \Omega,S'\rangle$ in only the point $S$, which is a tangent $k$-space through $S$ to $K$.

Conversely, if $K$ is a minimal blocking set, every $k$-space in $\PG(n,q)$, and hence in $\Gamma$ is blocked by the points of $K$, hence, $\bar{B}$ is a blocking set with respect to $k$-spaces. Since $K$ is minimal, there exists a tangent $k$-space $T_P$ to $K$ through every point $P$ of $\bar{B}$. It is clear that the projection of $T_P$ from $\Omega$ onto $\Gamma$ is a tangent $k$-space through $P$ to $\bar{B}$, hence, $\bar{B}$ is minimal.
\end{proof}
\begin{remarkN} The cone $K$ with vertex $\Omega$, where $\Omega$ is an $s$-dimensional subspace of $\PG(n,q)$, and base $\bar{B}$ (contained in an $(n-s-1)$-space $\Gamma$, skew from $\Omega$) has size $q^{s+1}|\bar{B}|+\frac{q^{s+1}-1}{q-1}$.

\end{remarkN}

\subsection{Linear blocking sets and the MPS-construction}
The following lemma is essentially a trivial observation, but it is the key idea behind the constructions presented in this paper.

\begin{lemma}\label{triviaal} Let $B'$ be a blocking set with respect to $(kt-1)$-spaces in $\PG(nt-1,q)$, then $B=\B(B')$ is a blocking set with respect to $(k-1)$-spaces in $\PG(n-1,q^t)$.
\end{lemma}
\begin{proof} As $B'$ blocks all $(kt-1)$-spaces in $\PG(nt-1,q)$, it also blocks the $(kt-1)$-spaces spanned by spread elements (i.e. the $\D_{k-1}$-spaces), which means that all $(k-1)$-spaces of $\PG(n-1,q^t)$ contain at least one point of $\B(B')$.
\end{proof}
The previous lemma provides us with a way of creating blocking sets $B$ in $\PG(n-1,q^t)$, using blocking sets $B'$ in $\PG(nt-1,q)$. An important problem is to determine when the obtained blocking set $B$ is minimal. In the MPS-construction and the construction of Costa, particular minimal blocking sets $B'$ are considered to ensure that $\B(B')$ is minimal. In the next subsections, we recall these constructions and translate them from the Barlotti-Cofman setting to the setting using field reduction which enables an easier description. 
Since {\em linear} blocking sets also fit in this framework and will be used later in this paper, we discuss them here.

\subsubsection{Linear blocking sets}
Linear blocking sets with respect to $(k-1)$-spaces in $\PG(n-1,q^t)$ were introduced by Lunardon \cite{L1}: he argues that an $\F_q$-linear set of rank $nt-kt+1$ is a blocking set. In view of Lemma \ref{triviaal}, this is clear: we take $B'$ to be an $(nt-kt)$-dimensional subspace of $\PG(nt-1,q)$, which is a blocking set with respect to $(kt-1)$-spaces, to obtain a (linear) blocking set $\B(B')$.

The importance of this construction lies in the fact that it provided counterexamples to the belief that all small minimal blocking set were of R\'edei-type \cite{polito2}. We will come back to this in Section \ref{linconj}.

It is important to note that linear blocking sets are necessarily minimal blocking sets. This was shown in \cite{L2} for $\F_q$-linear blocking sets with respect to lines in $\PG(n-1,q^t)$. But for blocking sets with respect to $(k-1)$-spaces in $\PG(n-1,q^t)$, $k\neq 2$, the minimality is up to our knowledge not directly proven in the literature; it does however follow easily from the following lemma of Sz\H{o}nyi and Weiner.

\begin{lemma} \cite[Lemma 3.1]{sz} Let $B$ be a blocking set of $\PG(n-1, q)$ with respect to $(k-1)$-dimensional subspaces, $q=p^h$, $p$ prime, and suppose that $|B|\leq 2q^{n-k}$. Assume that each $(k-1)$-dimensional subspace of $\PG(n, q)$ intersects $B$ in $1$ mod $p$ points. Then $B$ is minimal.
\end{lemma}
Since every subspace meets an $\F_q$-linear blocking set in $1$ mod $q$ points, and an $\F_q$-linear blocking set with respect to $(k-1)$-spaces in $\PG(n-1,q^t)$ has size at most $(q^{nt-kt+1}-1)/(q-1)$, this implies the following.
\begin{corollary}\label{linmin} An $\F_q$-linear blocking set with respect to $(k-1)$-spaces in $\PG(n-1,q^t)$ is minimal.
\end{corollary}

The fact that linear blocking sets are minimal will follow directly from Theorem \ref{hoofd1}.

\subsubsection{The MPS-construction}\label{MPS}
The MPS-construction goes as follows. Let $\S$ be a Desarguesian $(t-1)$-spread of $\Sigma=\PG(nt-t-1,q)$, embed $\Sigma$ as a hyperplane in $\Sigma'=\PG(nt-t,q)$ and consider the Barlotti-Cofman representation of $\PG(n-1,q^t)$ as $\Pi_{n-1}(\Sigma',\Sigma,\S)$ as described in Subsection \ref{barlotti}. Let $Y$ be a fixed element of $\S$ and let $\Omega$ be a hyperplane of $Y$. Let $\Gamma'$ be an $(nt-2t+1)$-dimensional subspace of $\Sigma'$, disjoint from $\Omega$ and denote by $\Gamma$ the $(nt-2t)$-dimensional subspace intersection of $\Gamma'$ and $\Sigma$ and by $T$ the intersection point of $\Gamma$ and $Y$. Let $\bar{B}$ be a blocking set of $\Gamma'$ such that $\bar{B}\cap \Gamma=\{Q\}$, $Q$ a point, with the property that $\ell\setminus \{T\}\not \subset \bar{B}$, for every line $\ell$ of $\Gamma'$ through $T$. Denote by $K$ the cone with vertex $\Omega$ and base $\bar{B}$. Let $B$ be the subset of $\Pi_{n-1}$ defined by $$B=(K\setminus \Sigma)\cup \{X\in \S: X\cap K\neq \emptyset\}.$$

\begin{figure}[h]
\begin{center}
\begin{tikzpicture}[scale=0.8]

\draw (0,2) ellipse (4cm and 2cm);
\draw(-4,2)  arc  (180:360:4cm and 4cm);

\draw (-3,2) ellipse (0.6cm and 1cm);
\draw (-1,2.8) ellipse (0.6cm and 1cm);
\draw (1,1.2) ellipse (0.6cm and 1cm);
\draw (3,2) ellipse (0.6cm and 1cm);
\draw[rotate around={45:(-1.7,0.2)}] (-1.7,0.2) ellipse (1.2cm and 1.7cm);
\draw (-1.3,-0.5) ellipse (0.4cm and 0.7cm);
\draw[thick,dashed] (-3.2,1.6)--(-1.6,-1);

\draw[thick,dashed] (-2.7,2.7)--(-1,0);

\draw (-3,2.2) ellipse (0.4cm and 0.7cm);
\node at (2.8,3.8) {  $\Sigma$};
\node at (2,-1) {  $\Sigma'$};
\node at (-3,2) { $\Omega$};

\node at (-3.5,0.7) { $T$};

\node at (-1.3,-0.5) { ${\bar{B}}$};
\node at (-0.2,0.7) { $\Gamma$};
\node at (0,-0.5) { $\Gamma'$};
\node[fill=black,circle,inner sep=1pt] at (-3,1.1) {};

\end{tikzpicture}
\caption{The MPS-construction}
\end{center}
\end{figure}
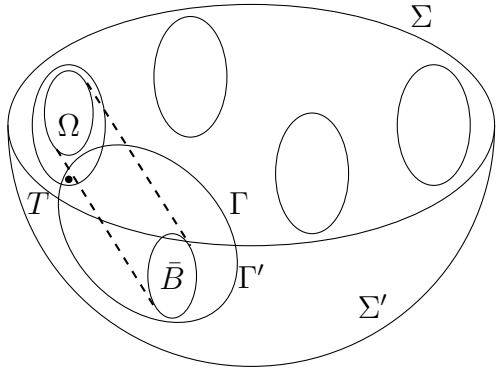

With these definitions, the authors show:
\begin{theorem}{\rm \cite[Proposition 1]{Leo}}\label{MPS1} The set $B$ is a blocking set of the projective space $\PG(n-1,q^t)$.
\end{theorem}

Let $B$ be the blocking set obtained by the MPS-construction, then, by using the correspondence between the Barlotti-Cofman-representation and the representation using field reduction (see Subsection \ref{barlotti}), we see that $B$ corresponds to $\B(K)$ by using the above definitions. We also see that the set $K$ is a blocking set with respect to $(nt-t-1)$-spaces in $\PG(nt-1,q)$: by Lemma \ref{kegel}, $K$ is a blocking set w.r.t. $(nt-2t)$-spaces in $\langle \Omega,\Gamma'\rangle$. Clearly every $(nt-t-1)$-space of $\PG(nt-1,q)$ meets $\langle \Omega,\Gamma'\rangle$ in a space of dimension at least $nt-2t$.

The authors distinguish two different cases of their construction: the case $T=Q$ (which they call Construction A) and the case $T\neq Q$ (Construction B). 

The authors show for both cases:
\begin{theorem}{\rm \cite[Proposition 2]{Leo}} \label{hoofdLeo}$B$ is a minimal blocking set of $\PG(n-1,q^t)$ if and only if $\bar{B}$ is a minimal blocking set of $\Gamma'$.
 \end{theorem}


\begin{remarkN} The construction of Costa \cite{costa} generalises Construction A as follows. 

The notations $\S$, $\Sigma$, $\Sigma'$ and $Y$ are used as before. Let $\Omega$ be an $s$-dimensional subspace of $Y$, with $0\leq s\leq t-2$, let $\Gamma'$ be an $(nt-t-s-1)$-dimensional subspace of $\Sigma$ skew from $\Omega$, let $\Gamma$ be the $(nt-t-s-2)$-dimensional intersection of $\Gamma'$ and $\Sigma$, let $\theta$ be the $(t-s-2)$-space $Y\cap \Gamma$. Define $R$ to be the set $\{\langle S,\Omega\rangle\cap \Gamma'$: $S$ is a hyperplane of $\Pi_{n-1}$, not containing $Y\}$. Let $\bar{B}$ be a blocking set with respect to elements of the set $R$, containing the space $\theta$. In the same way as before, $B$ is defined to be the set of points of the cone with vertex $\Omega$  and base $\bar{B}$, contained in $\Sigma'\setminus \Sigma$,  together with the point $Y$. 

Costa shows that $B$ is a minimal blocking set in $\Pi_{n-1}=\PG(n-1,q^t)$ if and only if $\bar{B}$ is a minimal blocking set with respect to elements of $R$.

\end{remarkN}

\begin{remarkN} The MPS-construction and the construction of Costa generalise the cone construction of Lemma \ref{kegel} in some sense. Instead of considering the cone with base a blocking set over $\F_{q^t}$, these constructions use cones over blocking sets over a subfield $\F_{q^t}$, i.e., a line through a point of the vertex and the base is no longer a full line, but a subline. The same idea was used for planar blocking sets before in \cite{pol} and \cite{large}.
\end{remarkN}

\section{Constructing minimal blocking sets with respect to $(k-1)$-spaces in $\PG(n-1,q^t)$} \label{S4}
\subsection{A general construction method}
\begin{construction} \label{con1} Let $\Omega$ be an $(nt-kt-2)$-dimensional subspace of $\PG(nt-1,q)$, let $\bar{B}$ be a blocking set, contained in a plane $\Gamma$ which is skew from $\Omega$ and let $K$ be the cone in $\Pi=\langle \Omega,\Gamma\rangle$ with vertex $\Omega$ and base $\bar{B}$. Let $B=\B(K)$.
\end{construction}

\begin{lemma} The set $B$ of Construction \ref{con1} is a blocking set with respect to $(k-1)$-spaces in $\PG(n-1,q^t)$.
\end{lemma}

\begin{proof} By Lemma \ref{kegel}, the cone with base $\bar{B}$ and vertex $\Omega$ is a blocking set with respect to lines in $\langle \Omega,\Gamma\rangle$, hence, with respect to $(kt-1)$-spaces in $\PG(nt-1,q)$. So the statement follows from Lemma \ref{triviaal}.
\end{proof}


\begin{lemma}\label{tangentspace} Let $\Pi$ be an $(nt-kt+1)$-dimensional subspace of $\PG(nt-1,q)$, $k\geq 2$. Let $\D$ be a Desarguesian $(t-1)$-spread in $\PG(nt-1,q)$ and let $P$ be a point of $\Omega$. Then there exists a $\D_{k-2}$-space meeting $\Omega$ only in points of $\B(P)$. Moreover, if $\dim(\B(P)\cap \Omega)\geq 1$, i.e. if the spread element through $P$ meets $\Omega$ in a subspace of dimension at least $1$, then there is a $\D_{k-1}$-space meeting $\Omega$ only in points of $\B(P)$.
\end{lemma}

\begin{proof} We will first show that there exists a $\D_s$-space containing only points of $\B(P)$ of $\Pi$, where $0 \leq s\leq k-2$. The statement holds trivially for $s=0$. Suppose that the statement holds for some $0<i\leq k-3$ and let $D$ be the obtained $\D_i$-space. The number of $\D_{i+1}$-spaces through $D$ is $\frac{(q^t)^{n-i-1}-1}{q^t-1}$. If $\B(P)$ meets $\Pi$ in a space of dimension $r$, then the number of $\D_{i+1}$-spaces through $D$ that meet $\Pi$ in a point, not belonging to $\B(P)$ is at most $\frac{q^{nt-kt-r+1}-1}{q-1}$, since different $\D_{i+1}$-spaces through $D$ that meet $\Pi$ meet in spaces that have only the points of $\B(P)$ in common. Now $\frac{q^{nt-kt-r+1}-1}{q-1}<\frac{(q^t)^{n-i-1}-1}{q^t-1}$ for all $r\geq 0$, and $i\leq k-3$, hence, at least one of the $\D_{i+1}$-spaces through $D$ meets $\Pi$ only in points of $\B(P)$. By induction we find a $\D_{k-2}$-space $T$ containing only points of $\B(P)$ of $\Pi$.

If $\B(P)$ meets $\Pi$ in a space of dimension at least one, then the above count with $r\geq 1$ and $i=k-2$ shows that there exists a $\D_{k-1}$-space through $T$ meeting $\Pi$ only in $\B(P)$.
\end{proof}

\begin{lemma}\label{HP} Let $\Omega$ be an $(nt-kt-2)$-dimensional subspace of $\PG(nt-1,q)$, let $\bar{B}$ be a minimal blocking set, contained in the plane $\Gamma$ which is skew from $\Omega$. Let $K$ be the cone in $\Pi=\langle \Omega,\Gamma\rangle$ with vertex $\Omega$ and base $\bar{B}$. Suppose that every point of $\bar{B}$ lies on at least $t$ tangent lines to $\bar{B}$ in $\Gamma$. If $P$ is a point of $K$, then $P$ lies in at least $t$ hyperplanes $H_P$ of $\Pi$ that meet $K$ only in some fixed subspace $\mu$ of $\Pi$ of codimension $2$.
\end{lemma}
\begin{proof} Let $P$ be a point of $K$, not in $\Omega$, and let $P'$ be the point $\langle \Omega,P\rangle\cap \Gamma$, which is contained in $\bar{B}$. By assumption, there are at least $t$ tangent lines $\ell_i^{P'}$, $i=1,\ldots,t,\ldots$ through $P'$ to $\bar{B}$. The space $\langle \ell_i^{P'},\Omega\rangle$ is a hyperplane of $\Pi$ which meets $K$ only in the space $\langle \Omega,P\rangle$ which has codimension $2$ in $\Pi$.

Let $P$ be a point of $\Omega$. By choosing all hyperplanes $\langle \Omega,\ell_i^{Q}\rangle$ with $Q$ arbitrary in $\bar{B}$ we satisfy the required condition.
\end{proof}

\begin{theorem} \label{hoofd1} If $\bar{B}$ is a minimal blocking set in $\Gamma$ such that every point of $\bar{B}$ lies on at least $2$ tangent lines to $\bar{B}$, then the blocking set $B=\B(K)$ obtained from Construction \ref{con1} is minimal.
\end{theorem}
\begin{proof}  Let $P$ be a point of $K$. We need to show that there is a $\D_{k-1}$-space through $P$ containing only points of $\B(P)$ of $K$. 
%
%
If $\B(P)$ meets $\Pi$ in a space of dimension at least one, then by Lemma \ref{tangentspace}, there exists a $\D_{k-1}$-space meeting $\Pi$ and hence $K$ only in $\B(P)$, which implies that there is a tangent $(k-1)$-space through $\B(P)$ to $B$ in $\PG(n-1,q^t)$.

So from now on, we suppose that $\B(P)$ meets $\Pi$ only in the point $P$. By Lemma \ref{tangentspace}, we have that there exists a $\D_{k-2}$-space $T$ meeting $\Pi$ only in the point $P$. By Lemma \ref{HP}, there are (at least) two hyperplanes, say $H_1$ and $H_2$, of $\Pi$ through $P$ which meet $K$ only in a subspace $\mu$ of $\Pi$ of codimension $2$.

Consider the quotient $\Pi/T$ in the quotient space $\PG(nt-1,q)/T\cong \PG(nt-kt+t-1,q)$. Note that, as $T$ is spanned by spread elements of the Desarguesian spread $\D$, $\D$ induces a Desarguesian $(t-1)$-spread $\D'$ in $\PG(nt-1,q)/T$. The $\D_{k-1}$-spaces through $T$ are in one-to-one correspondence with the elements of $\D'$. Since a tangent $\D_{k-1}$-space through $P$ to $K$ corresponds to an element of $\D'$ which meets $\Pi/T$ in a subspace skew to $K/T$, we need to show that there exists an element of $\D'$ meeting $\Pi/T$ in a subspace skew to $K/T$. Note that $\Pi/T$ is $(nt-kt)$-dimensional and that $H_1/T$ and $H_2/T$ are hyperplanes of $\Pi/T$ through $\mu/T$. 

Let $A$ be the number of elements of $\D'$ meeting $\Pi/T$ in a point, then expressing that $A+(\frac{q^{nt-kt+t}-1}{q^t-1}-A)(q+1)$ is at most $\frac{q^{nt-kt+1}-1}{q-1}$, the number of points in $\Pi/T$, yields that $A$ is at least $\frac{q+1}{q}(\frac{q^{nt-kt+t}-1}{q^t-1})-\frac{q^{nt-kt+1}-1}{q(q-1)}$. This implies that there are at most $\frac{q^{nt-kt+1}-1}{q-1}-A<2q^{nt-kt-1}$ points of $\Pi/T$ that are not the exact intersection of an element of $\D'$ with $\Pi/T$. This implies that there is at least one point of $H_1/T$ or $H_2/T$ that induces a tangent $\D_{k-1}$-space through $T$ which in turn implies that we have found a tangent $(k-1)$-space to $B$ in the point $\B(P)$ of $\PG(n-1,q^t)$ and that $P$ is essential.
%
%
%
\end{proof}

\begin{remarkN} By \cite{Br,Jamison}, an affine blocking set contains at least $2q-1$ points. This implies that every point of a minimal blocking set $\bar{B}$ in $\PG(2,q)$ with $|\bar{B}|=q+k, k\leq q$, lies on at least $q-k+1$ tangent lines to $\bar{B}$. So every minimal blocking set of size at most $2q-1$ satisfies the condition of Construction \ref{con1}. In particular, if $\bar{B}$ is a line, we find that, as announced before, a linear blocking set is minimal.
\end{remarkN}


Using that the blocking set constructed by Construction \ref{con1} is contained in $\B(\Pi)$, where $\Pi$ is an $(nt-kt+1)$-dimensional subspace of $\PG(nt-1,q)$, the following corollary easily follows.
\begin{corollary} There exists a minimal $(k-1)$-blocking set in $\PG(n-1,q^t)$ constructed by Construction \ref{con1} which spans an $s$-dimensional space for all $n-k\leq s\leq \min\{n-1,nt-kt+1\}$.
\end{corollary}

\begin{remarkN}If we compare Theorem \ref{hoofd1} to Theorem \ref{hoofdLeo} of Mazzocca, Polverino and Storme, we see that in the latter theorem the vice versa part also holds: if $B$ is minimal, than $\bar{B}$ is necessarily minimal. In general this does not hold for our construction. \end{remarkN}

\subsection{Construction \ref{con1} in a scattered subspace}
 In general, since the position of the space $\Pi=\langle \Omega,\Gamma\rangle$ with respect to the Desarguesian spread $\D$ is arbitrary, we cannot derive the size of the blocking set $B$ from the size of the blocking set $\bar{B}$. But if we take the space $\Pi$ to be e.g. a {\em scattered} subspace with respect to $\D$ (i.e. every element of $\D$ that meets this subspace, meets it in exactly $1$ point), we are able to derive the size of $B$ in terms of $|\bar{B}|$. Note that this is a restriction: it is clear that not all examples arising from Construction \ref{con1} can be obtained from a scattered subspace $\Pi$. In some cases, it is even impossible to find a scattered subspace of the right dimension, in view of the following theorem:

\begin{theorem}\label{max}\cite[Theorem 4.3]{BlLa2000}
A scattered $\F_q$-linear set in $\PG(n-1,q^t)$ has rank $\leq nt/2$.
\end{theorem}
In \cite[Theorem 2.5.5]{thesis}, Lavrauw shows the following:
\begin{theorem}\label{sc} If $r$ is even, then there exists a scattered subspace (with respect to a Desarguesian $(t-1)$-spread) of dimension $rt/2-1$ in $\PG(rt-1,q)$. If $r$ is odd, there exists a scattered subspace (with respect to a Desarguesian $(t-1)$-spread) of dimension $(rt-t)/2-1$.\end{theorem}

\begin{theorem} Let $k\geq (n+3)/2$. If $\bar{B}$ is a minimal blocking set in $\PG(2,q)$ such that every point of $\bar{B}$ lies on at least $2$ tangent lines to $\bar{B}$, then there exists a minimal blocking set with respect to $(k-1)$-spaces in $\PG(n-1,q^t)$ with size $|\bar{B}|q^{nt-kt-1}+\frac{q^{nt-kt-1}-1}{q-1}$.
\end{theorem}
\begin{proof} Consider an $(nt-kt+1)$-dimensional space $\Pi=\langle \Omega,\Gamma\rangle$, where $\Omega$ is $(nt-kt-1)$-dimensional and $\Gamma$ a plane skew from $\Omega$ and such that $\Pi$ is scattered with respect to $\D$, which is possible in view of Theorem \ref{sc} since $k\geq (n+3)/2$. Apply Construction \ref{con1} with the spaces $\Omega,\Gamma$ and the minimal blocking set $\bar{B}$ in $\Gamma$. Then, by Theorem \ref{hoofd1}, $\B(K)$ is a minimal blocking set with respect to $(k-1)$-spaces. Moreover, since every element of $\D$ that meets $\langle \Omega,\Gamma\rangle$, meets it in exactly one point, the number of points in $B$ is equal to the number of points in $K$, which is equal to $|\bar{B}|q^{nt-kt-1}+\frac{q^{nt-kt-1}-1}{q-1}$.
\end{proof}

The following corollary shows which are the possible dimensions spanned by a blocking set obtained by Construction \ref{con1} in a scattered subspace, again using Theorem \ref{sc}. 
\begin{corollary} If $\bar{B}$ is a minimal blocking set in $\PG(2,q)$ such that every point of $\bar{B}$ lies on at least $2$ tangent lines to $\bar{B}$ and if $nt-kt+1\leq \frac{n't-t}{2}-1$ and $n'-1\leq nt-kt+1$, then there is a $(k-1)$-blocking set $B$ in $\PG(n-1,q^t)$ spanning an $(n'-1)$-space, where $B$ has size $|\bar{B}|q^{nt-kt-1}+\frac{q^{nt-kt-1}-1}{q-1}$.
\end{corollary}

\subsection{A modification of Construction 1}
A disadvantage of Construction 1 is clearly the requirement that every point lies on at least two tangent lines to $\bar{B}$. Using the computer package `FinInG' for GAP \cite{GAP}, we found that we cannot remove this condition in general: there are $(t+1)$-subspaces in $\PG(3t-1,q)$ such that a cone over a Hermitian curve does not define a minimal blocking set in $\PG(2,q^t)$. In this subsection, we choose a particular subspace of $\PG(nt-1,q)$ so that we do not need at least two tangent lines to $\bar{B}$ in order to have minimality. 
\begin{figure}[h]
\begin{center}
\begin{tikzpicture}[scale=0.7]
\draw (0,0) ellipse (6cm and 4cm);
\draw (0,1) ellipse (4cm and 2cm);
\draw[thick] (-2,-0.5) ellipse (2cm and 2.5cm);
\draw[thick] (2,-2) ellipse (0.5cm and -0.8cm);
\draw (3,-3)--(2,0.5)--(1,-3)--cycle;
\draw[thick,dashed] (-2,-3)--(2,-2.8);
\draw[thick,dashed] (2.2,-1.27)--(-0.9, 1.6);

\node at (5,3) {  $\Pi$};
\node at (2,3) {  $\nu$};
\node at (-3.5,-1) { $\Omega$};
\node at (3,-2.1) {  $\Gamma$};
\node at (2,-2.1) { $\bar{B}$};
\node at (0,-2.5) {  $\mathbf{K}$};
\end{tikzpicture}

\end{center}
\caption{Construction 2}
\end{figure}
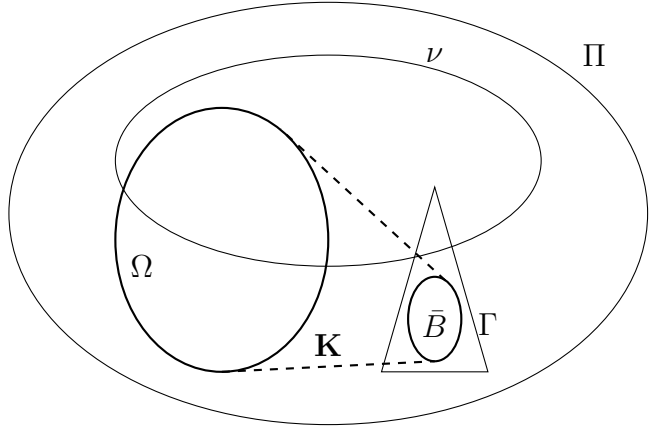

\begin{construction} \label{con2}Let $t\geq 4$. Let $\nu$ be a $\D_{n-k-1}$-space in $\PG(nt-1,q)$ and let $\Pi$ be an $(nt-kt+1)$-space through $\nu$ such that $\dim(\langle\B(\Pi)\rangle)=n-k+1$.  Let $\Omega$ be an $(nt-kt-2)$-dimensional subspace of $\Pi$, meeting $\nu$ in an $(nt-kt-4)$-dimensional space, let $\bar{B}$ be a minimal blocking set, contained in the plane $\Gamma$ of $\Pi$, which is skew from $\Omega$ and such that $\bar{B}$ is skew from $\nu$. Let $K$ be the cone with vertex $\Omega$ and base $\bar{B}$ and let $B=\B(K)$.
\end{construction}

\begin{theorem} The set $B$ from Construction 2 is a minimal $(k-1)$-blocking set in $\PG(n-1,q^t)$.
\end{theorem}
\begin{proof} The fact that $B$ is a $(k-1)$-blocking set follows from Lemma \ref{triviaal}.

The set $\B(\Pi)$ in $\PG(n-1,q^t)$ is a cone $C$ with vertex an $(n-k-1)$-space $\B(\nu)$ and base an $\F_q$-subline $\ell$. In the quotient space $\PG(n-1,q^t)/\B(\nu)$ we easily find a $(k-2)$-space skew from $\ell/\B(\nu)$, which forces the existence of a $(k-2)$-space $S$ skew from $C$. Every point $P$ in $K\cap \nu$ corresponds to a point $\B(P)$ of $B$ lying on the tangent $(k-1)$-space $\langle \B(\mu), S\rangle$, which means that the point $\B(P)$ is essential if $P\in \nu\cap K$.

Now let $P$ be a point of $K$, not contained in $\nu$. As in the proof of Theorem \ref{hoofd1}, we may restrict ourselves to the case where $\B(P)\cap \Pi=\{P\}$ and we know that there exists a $\D_{k-2}$-space $T$ through $\B(P)$ such that $T$ meets $\Pi$ only in the point $P$. We have that the quotient space $\Pi/T$ in $\PG(nt-1,q)/T\cong \PG(nt-kt+t-1,q)$ is $(nt-kt)$-dimensional and that $\D$ induces a Desarguesian spread $\D'$ in $\PG(nt-kt+t-1,q)$. Since $\Pi/T$ contains a $\D_{n-k-1}$-space $\nu/T$, every element of $\D'$ not meeting $\nu/T$ meets $\Pi/T$ in a single point. By Lemma \ref{HP}, since $\bar{B}$ is minimal, $P$ lies on at least one hyperplane $H$ of $\Pi$ meeting $K$ only in a space of codimension $2$. 

Since $\Omega$ meets $\nu$ in an $(nt-kt-4)$-space, the $(nt-kt-1)$-space $H/T$, is different from the $(nt-kt-1)$-space $\nu/T$. So we can consider a point $Q$ of $H/T$ in $\Pi/T$ which is not contained in $K$ nor $\nu/T$. This point $Q$ corresponds to a spread element $\D'$ which together with $T$ induces a tangent $\D_{k-1}$-space to $\B(P)$ in $\PG(n-1,q^t)$, which shows that $P$ is essential.
\end{proof}
Note that by construction, $\B(K)$ contains an $(n-k)$-space.

\begin{proposition} The $(k-1)$-blocking set $B$ in $\PG(n-1,q^t)$ obtained from Construction \ref{con2} spans a space of dimension $(n-k+1)$.
\end{proposition}
\begin{proof} The set $B$ is contained in the subspace $\langle \B(\Pi)\rangle$, which spans an $(n-k+1)$-dimensional space. Note that it is by construction not possible that $\langle \B(\Pi)\rangle=n-k+1$ and $\langle K\rangle=n-k$.
\end{proof}

\begin{proposition} Let $\Pi$ be so that $\langle \B(\Pi)\rangle$ is $(n-k+1)$-dimensional. The $(k-1)$-blocking set $B$ of $\PG(n-1,q^t)$ has size $$|\bar{B}|(q^{nt-kt-1}-q^{nt-kt-3})+q^{nt-kt-2}+q^{nt-kt-3}+\epsilon,$$
where $\epsilon$ is at least $1$.
\end{proposition}
\begin{proof} Every spread element of $\D$ meeting $\Pi$ is either entirely contained in $\Pi$ or meets $\Pi$ in a point. The cone $K$ consists of a union of $(nt-kt-1)$-spaces through $\Omega$. Each of these $(nt-kt-1)$-spaces contains $q^{nt-kt-1}$ points, not in $\Omega$ of which, since $\bar{B}$ does not meet $\nu$, $q^{nt-kt-1}-q^{nt-kt-3}$ are not contained in $\nu$. There are $q^{nt-kt-2}+q^{nt-kt-3}$ points in $\Omega\setminus \nu$. The only points of $B$ that we have not counted yet are the points of $K$ in $\nu$. It is clear that, since $\Omega$ meets $\nu$ non-trivially, this number is at least $1$.
\end{proof}
Applying the previous proposition for $k=n-1$, and using that a $\D_{n-k-1}$-space is a spread element, we find the following corollary. The second part follows from the fact that a Hermitian curve in $\PG(2,q)$, $q$ square, is a minimal blocking set of size $q\sqrt{q}+1$.

\begin{corollary} \label{crr}The blocking set $B$ in $\PG(n-1,q^t)$ obtained from a planar blocking set $\bar{B}$ by Construction \ref{con2} has size $$|\bar{B}|(q^{t-1}-q^{t-3})+q^{t-2}+q^{t-3}+1.$$ In particular, if $q$ is a square, we find minimal blocking sets in $\PG(2,q^t)$ of size $q^t\sqrt{q}+q^{t-1}-q^{t-2}\sqrt{q}+q^{t-2}+1$.
\end{corollary}
In \cite{pol}, Mazzocca and Polverino construct blocking sets in $\PG(2,q^t)$ arising from cones in $\PG(2t,q)$. We want to point out that, starting from a Hermitian curve $\bar{B}$ in $\PG(2,q)$, $q$ square, they construct minimal blocking sets in $\PG(2,q^t)$ of size $q^{t-1}(|\bar{B}|-1)+1=q^{t}\sqrt{q}+1$, which is smaller than the size of the minimal blocking set obtained from a Hermitian curve in the previous corollary.

\section{Minimal blocking sets with respect to the hyperplanes}\label{hyper}

\subsection{The main observation}
We have seen in Lemma \ref{triviaal} that, if $B'$ is a blocking set with respect to $(nt-t-1)$-spaces in $\PG(nt-1,q)$, then $B=\B(B')$ is a blocking set (w.r.t. hyperplanes) in $\PG(n-1,q^t)$. In the following theorem, we show that a kind of reverse statement also holds.

\begin{theorem} \label{obs} Let $B$ be a minimal blocking set with respect to the hyperplanes of $\PG(n-1,q^t)$, then $B$ can be written as $\B(B')$, where $B'$ is a minimal blocking set with respect to $(nt-t-1)$-spaces of $\PG(nt-1,q)$.
\end{theorem}
\begin{proof} Let $\S(B)$ denote the set of spread elements corresponding to the points of $B$ and let $\tilde{B}$ be the point set of the elements of $\S(B)$. Since $B$ is a blocking set with respect to the hyperplanes of $\PG(n-1,q^t)$, every $\D_{n-2}$-space contains an element of $\S(B)$, and hence, certainly a point of $\tilde{B}$ (in fact, at least $\frac{q^t-1}{q-1}$ of them).

Now consider an $(nt-t-1)$-space $\pi$ of $\PG(nt-1,q)$ which is not a $\D_{n-2}$-space. Let $H$ be a hyperplane of $\PG(nt-1,q)$ through $\pi$. We know from Lemma \ref{hyperspan} that $H$ contains a $\D_{n-2}$-space $\pi'$. Since $B$ is a blocking set with respect to the hyperplanes, the space $\pi'$ contains at least one element of $\S(B)$, say $S$. Now $\pi\cap \pi'$ is at least $(nt-2t)$-dimensional, hence, since $S$ is $(t-1)$-dimensional and contained in $\pi'$, $S$ meets $\pi$ non-trivially, and hence, $\pi$ contains at least one point of the set $\tilde{B}$. This implies that $\tilde{B}$ is a blocking set with respect to $(nt-t-1)$-spaces.

Now let $B'$ be a minimal blocking set with respect to $(nt-t-1)$-spaces contained in $\tilde{B}$ ($B'$ is not necessarily unique). To show that $\B(B')=B$, we show that in every element of $\S(B)$, there lies at least one point of $B'$. Suppose that there is an element of $\S(B)$, say $T$, that does not contain a point of $B'$. Since $B$ is a minimal blocking set, there is a tangent hyperplane in $\PG(n-1,q^t)$ to $B$ in the point corresponding to $T$. This tangent hyperplane corresponds to a $\D_{n-2}$-space (which is an $(nt-t-1)$-space) meeting $\S(B)$ in exactly the element $T$, hence, does not contain a point of the blocking set with respect to $(nt-t-1)$-spaces $B'$, a contradiction.
\end{proof}

\begin{remarkN} The blocking set $B'$ found in the previous theorem is not necessarily unique. Consider for example the blocking set $B$ in $\PG(2,q^t)$ consisting of all points of a line. Then $B$ corresponds to the set of spread elements in a $\D_1$-space $\pi$ of $\PG(3t-1,q)$. It is clear that $B=\B(\mu)$ for any subspace $\mu$ of dimension $t$ in $\pi$. Take a plane $\nu$ in $\pi$ and a $(2t-4)$-dimensional subspace $\nu'$ in $\pi$, skew from $\nu$, then the cone with vertex $\nu'$ and base a large minimal blocking set in $\nu$ (e.g. a Hermitian curve if $q$ is a square) is a minimal blocking set $B''$ with respect to $(t-1)$-spaces in $\pi$, hence, $B=\B(B'')$. Note that $\mu$ is a small minimal blocking set with respect to $(2t-1)$-spaces in $\PG(3t-1,q)$, whereas $B''$ is a large minimal blocking set with respect to $(2t-1)$-spaces in $\PG(3t-1,q)$.
\end{remarkN}

\begin{remarkN} Let $B$ be a Hermitian curve in $\PG(2,q^2)$. Then $B$ is a minimal blocking set of size $q^3+1$. If we apply field reduction to $B$, then we obtain a set $\mathcal{S}(B)$ of $q^3+1$ lines in $\PG(5,q)$ of which the point set $\tilde{B}$ blocks all $3$-spaces by Theorem \ref{obs}. The $(q^3+1)(q+1)$ points do not form a minimal blocking set with respect to $3$-spaces. It is not too hard to check that $\tilde{B}$ is the point set of an elliptic quadric in $\PG(5,q)$ (see e.g. \cite{LaVa13}) and that the $q^3+q^2+q+1$ points of a parabolic quadric $\Q(4,q)$ in $\tilde{B}$ form a minimal blocking set $B'$ with respect to $3$-spaces in $\PG(5,q)$, hence, $B=\B(B')$. But in general, it is not easy to construct a minimal blocking set $B'$ with respect to $(nt-t-1)$-spaces contained in the point set of the spread elements corresponding to $B$.
\end{remarkN}

\subsection{Small minimal blocking sets and the linearity conjecture} \label{linconj}
The {\em linearity conjecture} for small blocking sets states that all small minimal blocking sets in $\PG(n-1,q^t)$ are {\em linear sets} over $\F_p$, where $q^t=p^h$, $p$ prime (see \cite{sziklai}).

We have the following result of T. Sz\H{o}nyi and Zs. Weiner \cite{sz}, showing that the linearity conjecture holds for projective spaces over fields of prime order.
\begin{theorem}\cite[Corollary 3.15]{sz}\label{tamas} A small minimal blocking set with respect to $k$-spaces in $\PG(n,p)$, $p$ prime, is an $(n-k)$-space.
\end{theorem}
In general, the linearity conjecture for small minimal blocking sets remains open. For an overview of the cases in which the linearity conjecture has proven to be true, we refer to \cite{LaVa13}.

If $B$ is a small minimal blocking set with respect to the hyperplanes of $\PG(n-1,p^h)$, $p$ prime, by Theorem \ref{obs}, $B$ can be written as $\B(B')$ where $B'$ is a minimal blocking set with respect to $(hn-h-1)$-spaces of $\PG(nh-1,p)$. If this latter blocking set is small then by Theorem \ref{tamas}, it is an $h$-space. Hence, the linearity conjecture can be restated as follows (compare to the statement of Theorem \ref{obs}):\\

{\bf Linearity conjecture:}  Let $B$ be a small minimal blocking set with respect to the hyperplanes of $\PG(n-1,p^h)$, $p$ prime, then $B$ can be written as $\B(B')$, where $B'$ is a {\em small} minimal blocking set with respect to $(nh-h-1)$-spaces of $\PG(nh-1,p)$.\\


%

We will now investigate the properties of small minimal blocking sets that arise from the MPS-construction and from Construction \ref{con1}; we will show that the obtained blocking sets are linear blocking sets.

Define the {\em exponent} $e$ of a small minimal blocking set $B$ with respect to $k$-spaces as the largest integer such that every $k$-space meets $B$ in $1$ mod $p^e$ points. Lemma \ref{hulp}(1) will show that $e$ is well-defined.

We will need the following properties of small minimal blocking sets in $\PG(n,p^t)$, $p$ prime. For item (2), for simplicity, we state a slightly weaker bound than the one proven in \cite[Theorem 3.9]{sz}.
\begin{lemma}\label{hulp}
\begin{enumerate}
\item  \cite[Proposition 2.7]{sz} Every $k$-space meets a small minimal blocking set with respect to $k$-spaces in $\PG(n,p^t)$, $p$ prime, in $1$ mod $p$ points, hence, $e\geq 1$.
\item \cite[Theorem 3.9]{sz} A small minimal blocking set in $\PG(n,p^t)$, $p\geq 7$, with exponent $e$ has at most $p^t+2p^{t-e}+1$ points.
\item \cite[Corollary 3.11]{sz} Every subspace meets a small minimal blocking set with respect to $k$-spaces with exponent $e$ in $1$ mod $p^e$ or zero points.
\item  \cite[Lemma 6]{higher} Let $B$ be a small minimal blocking set with exponent $e$ in $\PG(n, p^t)$, $p$ prime. If
for a certain line $L$, $|L\cap B| = p^e + 1$, then $\F_{p^e}$ is a subfield of $\F_{p^t}$ and $L\cap B$ is $\F_{p^e}$-linear.
\item  \cite[Lemma 4]{higher} A point of a small minimal blocking set $B$ with exponent $e$ in $\PG(n,p^t)$, $p\geq 7$, $p$ prime, lying on a $(p^e+1)$-secant, lies on at least $p^{t-e}-4p^{t-e-1}+1$ $(p^e+1)$-secants.
\end{enumerate}
\end{lemma}
From Lemma \ref{hulp}(3) we easily deduce the following corollary which also reproves Theorem \ref{tamas}.
\begin{corollary} \label{t2}A small minimal blocking set with respect to $k$-spaces with exponent $t$ in $\PG(n,p^t)$, $p$ prime, is an $(n-k)$-space.
\end{corollary}

\begin{theorem} If a small minimal blocking set $B$ in $\PG(n-1,p^t)$, $p$ prime, arises from the MPS construction in $\PG(nt-1,p)$ then it is linear and of R\'edei-type.
\end{theorem}
\begin{proof} We use the notations of Section \ref{MPS}. Let $B$ be a small minimal blocking set with respect to the hyperplanes in $\PG(n-1,p^t)$ arising from the MPS-construction, then $B=\B(K)$, where $K$ is a blocking set with respect to the hyperplanes of $\Sigma'=\PG(nt-t,p)$ and $K$ meets $\Sigma$ in a $(t-1)$-space. Since all points of $B$, not in the hyperplane corresponding to the $\D_{n-2}$-space $\Sigma$ correspond to a unique point of $K$, and $\vert B\vert\leq p^t+p^{t-1}+1$ we have that $K$ has at most $p^t+2p^{t-1}+\frac{p^{t}-1}{p-1}$ points, which implies that $K$ is small. Since $K$ is a minimal blocking set by Theorem \ref{hoofdLeo}, it is a $t$-space by Theorem \ref{tamas}.

 It follows that $B$ is $\F_p$-linear. If the $(t-1)$-space $K\cap \Sigma$ is not contained in the spread element $Y$, then $\Sigma$ corresponds to a hyperplane containing $p^{t-1}+1=\vert B\vert-p^t$ points of $B$, hence, $B$ is of R\'edei-type. If $K\cap \Sigma$ equals $Y$, then let $Q$ be a point of $K$, not in $\Sigma$. It is clear that, since $K$ is a $t$-space through the spread element $Y$, $K$ is contained in the $\D_1$-space $\langle Y,\B(Y)\rangle$, hence $B$ is a line.
\end{proof}

\begin{remarkN} The vice versa part of the previous theorem does not hold: let $\pi$ be a $t$-space in $\PG(nt-1,p)$ such that there is an element $Y$ of $\D$ meeting $\pi$ in a $(t-3)$-space and an element $Z$ meeting $\pi$ in a line and such that $\pi$ is not entirely contained in $\langle Y,Z\rangle$. Then the line corresponding to the $\D_1$-space $\langle Y,Z\rangle$ clearly contains $\vert B\vert - p^t$ points of $\B(\pi)$, hence, the small linear blocking set $\B(\pi)$ is of R\'edei-type, but there does not exist an element of $\D$ meeting $\pi$ in a $(t-2)$-space, so $\B(\pi)$ cannot be constructed by the MPS-construction.
\end{remarkN}

By Lemma \ref{hulp}(4), the points of $B$ on a $(p^e+1)$-secant form an $\F_{p^e}$-linear set of rank $2$, i.e., a subline, which by field reduction corresponds to a regulus. We need the following information on the intersection of a regulus with a plane.
\begin{lemma}\cite[Lemma 6, Corollary 13]{lavrauw}\label{h33}
A plane $\pi$ meeting all elements of a regulus $\B(\ell)$, meets the point set of $\B(\ell)$ either in a line, or in two lines, or in a conic.
\end{lemma}
\begin{corollary} \label{h3} Suppose that $\B(\ell)$ where $\ell$ is a line, is contained in $\B(K)$, where $K$ is a blocking set with respect to lines in $\Pi$. Then the intersection of the point set of $\B(\ell)$ with a plane of $\Pi$ contains a line.
\end{corollary}
\begin{proof} From Lemma \ref{h33}, we get that the intersection of the point set of $\B(\ell)$ with a plane either contains a line or is a conic. But a conic does not block all lines in the plane $\pi$ and $K\cap \pi$ is a blocking set with respect to lines in $\pi$, hence, this possibility does not occur.
\end{proof}

\begin{lemma}\label{vlak} A plane $\pi$ meets the cone $K$, with vertex a $(t-2)$-dimensional space $\Omega$ and base a small minimal blocking set $\bar{B}$ in a plane $\Gamma$, skew from $\Omega$, either in the plane $\pi$ itself, in a number of lines through a fixed point or in a minimal blocking set equivalent to $\bar{B}$.
\end{lemma}
\begin{proof} Let $\pi$ be a plane in $\Pi=\langle \Omega,\Gamma\rangle$. 
We have the following possibilities:
\begin{itemize}\item $\pi$ is contained in $\Omega$. In this case, $\pi$ is contained in $K$.
\item $\pi$ meets $\Omega$ in a line $L$. If $\pi$ contains a point $P$ of $K$, not on $L$, then $\langle \Omega,P\rangle$  is contained in $K$, hence, $\pi$ is contained in $K$.
\item $\pi$ meets $\Omega$ in a point $P$.  Since $K$ is a blocking set with respect to lines in $\Pi$, $K$ forms a blocking set with respect to lines in the plane $\pi$, so there exists a point of $K$ in $\pi$, different from $P$. It is clear that every point of $K$ in $\pi$, different from $P$, gives rise to a line contained in $K$ through $P$.
\item $\pi$ is skew from $\Omega$. Consider the mapping $\phi$ from $\Gamma$ to $\pi$ defined by mapping the point $P$ of $\Gamma$ to the intersection of the cone $\langle P,\Omega\rangle$ with $\pi$. It is clear that $\phi$ defines a collineation mapping $\bar{B}$ onto the intersection of $K$ with $\pi$, hence, $K$ and $\pi$ intersect in a minimal blocking set, equivalent to $\bar{B}$.\qedhere\end{itemize}\end{proof}

\begin{lemma}\label{hulp2} Let $P$ be a point of a cone $K$ with vertex a $(t-2)$-dimensional space $\Omega$ and base a non-trivial minimal blocking set in a plane $\Gamma$ skew from $\Omega$ and suppose that $P$ is not contained in the vertex of $K$, then $P$ lies on $(q^{t-1}-1)/(q-1)$ lines contained in $K$.
\end{lemma}
\begin{proof}
Since the vertex $\Omega$ of the cone $K$ is $(t-2)$-dimensional, it is clear that a point $P$, not in $\Omega$, lies on $(q^{t-1}-1)/(q-1)$ lines of $\langle P,\Omega\rangle$, which are contained in $K$. If $P$ lies on another line $L$ that is contained in $K$, then this line is skew from $\Omega$ and $\langle\Omega,L\rangle$ meets $\Gamma$ in a line $L'$, contained in $K$, a contradiction since $K\cap \Gamma$ is a non-trivial minimal blocking set.
\end{proof}

We now extend a property of small minimal planar blocking sets (\cite[Proposition 4.17]{sziklai}) to blocking sets in $\PG(n,q)$, $n\geq 2$.

\begin{lemma} \label{h1} If a small minimal blocking set $B$ in $\PG(n,p_0^h)$ has exponent $e$, $p_0:=p^e\geq 7$, $p$ prime, then there exists a $(p_0+1)$-secant to $B$.
\end{lemma}
\begin{proof} We proceed by induction, where the case $n=2$ is Proposition 4.17 of \cite{sziklai}. Since $B$ has exponent $e$, there is a hyperplane $H$ with $|H\cap B|=1$ mod $p^e$ and $|H\cap B|\neq 1$ mod $p^{e+1}$. It is clear that, since every subspace meets $B$ in $1 \mod p^e$ points and the number of points in $B$ is equal to $1 \mod p^e$, that we can find a line $L$ in $H$ with $\lambda=1 \mod p^e$ points and  $\lambda \neq 1 \mod p^{e+1}$. 
If a plane contains a point of $B$ outside of $L$, then this plane contains at least $p^{2e}$ points of $B$ outside of $L$ since the line through $2$ points of $B$ contains at least $p^e+1$ points of $B$ by Lemma \ref{hulp}(3). By Lemma \ref{hulp}(1), this implies that there exists a plane $\pi$ through $L$ without extra points of $B$. Let $P$ be a point of $\pi$, not on $L$, then the projection $B'$ of $B$ from $P$ onto a hyperplane through $L$ and not through $P$ is a small minimal blocking set in $\PG(n-1,p_0^h)$ (see e.g. \cite[Corollary 3.2]{higher}), which has, by our claim, exponent $e$. So by induction, we find a $(p_0+1)$-secant to $B'$ in $H'$. Since all points of $B$ in $\pi$ are on the line $L$, we have found a $(p_0+1)$-secant to $B$.
\end{proof}

\begin{corollary} \label{lemma1} Let $B$ be a small minimal blocking set with exponent $e$ in $\PG(n,p_0^h)$, $p_0:=p^e\geq 7$, $p$ prime, then there are at least $(p_0^{h-1}-4p_0^{h-2}+1)p_0+1$ points of $B$ that each lie on at least $p_0^{h-1}-4p_0^{h-2}+1$ $(p_0+1)$-secants to $B$.
\end{corollary}
\begin{proof} From Lemma \ref{h1}, we get that there exists one $(p_0+1)$-secant to $B$. Lemma \ref{hulp} shows that through a point of this $(p_0+1)$-secant there are at least $p_0^{h-1}-4p_0^{h-2}+1$ $(p_0+1)$-secants. 
\end{proof}

%

As mentioned before, not all linear blocking sets are of R\'edei-type, and we will now show that a small minimal blocking set arises from Construction \ref{con1} if and only if it is linear. 

\begin{theorem}\label{final} If $p\geq 7$, a small minimal blocking set with respect to the hyperplanes in $\PG(n-1,q^t)$, $q=p^h$, $p$ prime, with exponent $e$ arises from Construction \ref{con1} in $\PG(nht/e-1,p^e)$ if and only if it is an $\F_{p^e}$-linear blocking set.
\end{theorem}

\begin{proof} Put $q_0=p^e$, $q^t=q_0^{t_0}$ (hence $t_0=ht/e$). Let $B$ be an $\F_{p^e}$-linear blocking set, then $B=\B_{\D'}(\pi)$, where $\pi$ is a $t_0$-space in $\PG(nt_0-1,q_0)$ and $\D'$ is a Desarguesian $(t_0-1)$-spread in $\PG(nt_0-1,q_0)$. Let $\Omega$ be a $(t_0-2)$-dimensional subspace of $\pi$ and let $\Gamma$ be a plane meeting $\pi$ in a line $\bar{B}$, disjoint from $\Omega$. It is clear that $B=\B_{D'}(K)$, with $K$ the cone with vertex $\Omega$ and base $\bar{B}$; note that $\bar{B}$ is a minimal blocking set such that every point of $\bar{B}$ lies on at least $2$ tangent lines to $\bar{B}$. So this implies that $B$ is obtainable from Construction \ref{con1}.

So now assume that $B$ is a small minimal blocking set obtained from Construction \ref{con1} for some $\Pi=\langle \Omega,\Gamma\rangle$ and $\bar{B}$ where $\bar{B}$ is a minimal blocking set in the plane $\Gamma$ in $\PG(nt_0-1,q_0)$, where $\Gamma$ is skew from the $(t_0-2)$-space $\Omega$. We will show that $B=\B_{D'}(\pi)$, where $\pi$ is a $t_0$-space in $\PG(nt-1,q_0)$, where $\D'$ is  a Desarguesian $(t_0-1)$-spread in $\PG(nt_0-1,q_0)$.
Corollary \ref{lemma1} states that there exists a point $P$ of $B$ on at least $q_0^{t_0-1}-4q_0^{t_0-2}+1$ $(q_0+1)$-secants, and such that $P$ is not an element of $\B(\Omega)$ since $\vert \B(\Omega)\vert \leq \frac{q_0^{t-1}-1}{q_0-1}$. Consider the spread element $\B_{\D'}(P)$ and its intersection $S$ with the space $\Pi=\langle \Omega,\Gamma\rangle$. Let $L_1,\ldots,L_r$ be the $(q_0+1)$-secants through $P$ to $B$. Every $L_i$ corresponds to a $(2t_0-1)$-dimensional space in $\PG(nt_0-1,q_0)$, denote by $\pi_1,\ldots,\pi_r$ the subspaces of $\Pi$ that occur as the intersection of $\Pi$ with the $(2t_0-1)$-dimensional space corresponding to $L_i$. Note that two different spaces $\pi_i$ and $\pi_j$ meet exactly in the subspace $S$. From this, it follows that if $S$ has dimension $s$, then the number $r$ can be at most the number of spaces of dimension $s+1$ through the $s$-dimensional subspace $S$ in $\Pi$. This number equals $(q_0^{t_0-s}-1)/(q_0-1)$, which is smaller than $q_0^{t_0-1}-4q_0^{t_0-2}+1$ if $s$ is at least $1$. So this implies that $s=0$.

%
%

So $\B_{D'}(P) \cap \langle \Omega,\Gamma\rangle$ is a point. By an easy counting, we find that there are more than $(q_0^{t_0-1}-1)/(q_0-1)$ of the spaces $\pi_i$ that are $1$-or $2$-dimensional. By Corollary \ref{h3} these all give rise to a line through $P$, contained in $K$. So by Lemma \ref{hulp2}, we find that $\bar{B}$ is a line which proves the statement.
\end{proof}

\end{document}